\documentclass[a4paper,11pt]{article}

\usepackage{amsmath,amssymb,amsthm,mathtools}
\usepackage{booktabs}
\usepackage{enumitem}
\usepackage{geometry}
\usepackage{hyperref}
\hypersetup{
  unicode=true,
  colorlinks=true,
  linkcolor=black,
  citecolor=black,
  urlcolor=blue,
  pdftitle={Brauer Relations for Symmetric Groups: Young--Wreath Quotients and the n-Cycle Mark Obstruction},
  pdfauthor={Wakatake Masahiro}
}

\numberwithin{equation}{section}

\newtheorem{theorem}{Theorem}[section]
\newtheorem{maintheorem}[theorem]{Main Theorem}
\newtheorem{proposition}[theorem]{Proposition}
\newtheorem{lemma}[theorem]{Lemma}
\newtheorem{corollary}[theorem]{Corollary}

\newtheorem{definition}[theorem]{Definition}
\newtheorem*{maintheoremrestated}{Main Theorem~\ref{thm:main}}
\newtheorem*{primitivecomparisonrestated}{Main Theorem~\ref{thm:primitive-comparison}}

\newcommand{\Q}{\mathbb{Q}}
\newcommand{\Z}{\mathbb{Z}}

\newcommand{\Ind}{\operatorname{Ind}}
\newcommand{\Inf}{\operatorname{Inf}}
\newcommand{\Res}{\operatorname{Res}}
\newcommand{\Prim}{\operatorname{Prim}}
\newcommand{\coeff}{\operatorname{coeff}}

\newcommand{\red}{\operatorname{red}}
\newcommand{\one}{\mathbf{1}}
\newcommand{\Hom}{\operatorname{Hom}}

\title{Brauer Relations for Symmetric Groups:\\
Young--Wreath Quotients and the \textit{n}-Cycle Mark Obstruction}
\author{Wakatake Masahiro}
\date{}

\begin{document}
\maketitle

\begin{abstract}
Let $K(S_n)$ be the kernel of the linearization map from the Burnside ring of the symmetric group $S_n$ to its rational representation ring.  We denote by $N_n$ the free $\mathbb Z$-sublattice with basis given by the standard Brauer relations $\Theta_H$, indexed by the $S_n$-conjugacy classes of subgroups that are either intransitive or transitive imprimitive and are not conjugate to Young subgroups.  We also let $J_n$ be the $\mathbb Z$-lattice of relations induced from the kernels of the linearization maps for proper Young subgroups and standard wreath subgroups $S_a\wr S_b$, and put $O_n=N_n/J_n$.  For every composite integer $n\ge4$, we construct, using the $C_n$-mark and the Young section, an integral homomorphism $\omega_n:N_n\to\Z$ and prove directly over $\mathbb Z$ that
\[
  J_n=\ker(\omega_n|_{N_n}),\qquad O_n\cong\Z.
\]
We further construct a natural comparison homomorphism from $O_n$ to the primitive quotient $\Prim(S_n)$, obtained by factoring out all imprimitive relations arising from proper subquotients.  Combined with the classification theorem of Bartel--Dokchitser, this shows that $O_n\xrightarrow{\sim}\Prim(S_n)$ for composite $n\ge6$, while for $n=4$ the comparison map is reduction modulo $2$ under the identifications $O_4\cong\Z$ and $\Prim(S_4)\cong\Z/2\Z$.  On the other hand, for prime $n\ge5$, one has $O_n=0$ whereas $\Prim(S_n)\cong\Z$.  Finally, we study the monomial Burnside ring over $C_2$ and show, via an additive map that we call index-two permutation reduction, that the resulting quotient is canonically isomorphic to the ordinary Young--wreath quotient.
\end{abstract}

\begin{center}
\end{center}

\section{Introduction}

The Burnside ring, the rational representation ring, the linearization map, and Brauer relations for a finite group $G$ are recalled in Section~\ref{sec:prelim}. We use the notation
\[
  \ell_G:\Omega(G)\longrightarrow R_{\Q}(G),\qquad
  K(G)=\ker\ell_G.
\]
Let $I(G)$ denote the lattice of imprimitive relations generated by induction from proper subgroups and inflation from proper quotient groups, and write
\[
  \Prim(G)=K(G)/I(G).
\]
Primitive relations and the groups $\Prim(G)$ for arbitrary finite groups were classified by Bartel--Dokchitser~\cite{BartelDokchitserBrauer}. The purpose of this paper is to show that, for symmetric groups, the primitive quotient can be described using only a Young--wreath induction family that is smaller than the family of all proper subquotients.

In Section~\ref{subsec:young-partial}, we construct the partial Burnside ring
\[
  P_n:=\Omega(S_n,\mathcal Y_n)
\]
associated with the family $\mathcal Y_n$ of Young subgroups, together with the inverse
\[
  \sigma_n:R_{\Q}(S_n)\xrightarrow{\sim}P_n
\]
of the restriction of the linearization map. To each subgroup $H\le S_n$, we associate the standard relation
\begin{equation}\label{eq:theta-intro}
  \Theta_H=[S_n/H]-\sigma_n\bigl([\Q[S_n/H]]\bigr)\in K(S_n).
\end{equation}

Suppose that the natural action of a subgroup $H\le S_n$ on $\{1,\ldots,n\}$ is either intransitive or transitive imprimitive. For brevity, we call such a subgroup $H$ \emph{action-imprimitive}, and denote the set of all such subgroups by $\mathcal H_n^{\mathrm{np}}$. This is terminology local to the present paper and is distinct from ``imprimitive'' as applied to Brauer relations.

By Proposition~\ref{prop:standard-basis}, the standard relations indexed by
non-Young conjugacy classes are linearly independent. We therefore set
\begin{equation}\label{eq:N-intro}
  N_n
  :=
  \bigoplus_{\substack{
    (H)\in\operatorname{Sub}(S_n)^c\\
    H\in\mathcal H_n^{\mathrm{np}}\\
    H\text{ is not conjugate to a Young subgroup}
  }}
  \Z\Theta_H
  \subseteq K(S_n).
\end{equation}
Thus the displayed elements $\Theta_H$ form a $\Z$-basis of $N_n$; in
particular, no zero relations or repeated conjugacy classes occur in this
presentation. For a factorization $n=ab$ with $a,b>1$, let
\[
  W_{a,b}=S_a\wr S_b=S_a^b\rtimes S_b\le S_n
\]
be the stabilizer of the standard system of $b$ blocks of size $a$, that is, the standard wreath subgroup associated with the natural imprimitive action. We further define
\begin{align}
 J_n^{\mathrm Y}
 &:=\sum_{\substack{\lambda\vdash n\\\lambda\ne(n)}}
 \Ind_{S_\lambda}^{S_n}K(S_\lambda),\\
 J_n^{\mathrm W}
 &:=\sum_{\substack{ab=n\\a,b>1}}
 \Ind_{W_{a,b}}^{S_n}K(W_{a,b}),
 \qquad J_n:=J_n^{\mathrm Y}+J_n^{\mathrm W}.
 \label{eq:J-intro}
\end{align}
We shall prove that $J_n\subseteq N_n$, and call
\[
  O_n:=N_n/J_n
\]
the \emph{Young--wreath quotient}.

Let $c=(1\ 2\ \cdots\ n)$ and $C_n=\langle c\rangle$. Among the fixed-point marks introduced in Section~\ref{subsec:marks}, we use
\[
 \phi_{C_n}:\Omega(S_n)\longrightarrow\Z,
 \qquad \phi_{C_n}([S_n/H])=|(S_n/H)^{C_n}|.
\]
Writing $\coeff_{[S_n/H]}(x)$ for the coefficient of $[S_n/H]$
in the Burnside basis expansion of $x$, we define
\begin{equation}\label{eq:omega-intro}
  \omega_n:K(S_n)\longrightarrow\Z,
  \qquad \omega_n(x)=-\coeff_{[S_n/S_n]}(x).
\end{equation}
In Proposition~\ref{prop:fixed-coset}, we prove that
$\omega_n(\Theta_H)=\phi_{C_n}([S_n/H])$. We call $\omega_n$ the \emph{$n$-cycle mark obstruction}.

\begin{maintheorem}\label{thm:main}
Let $n\ge4$ be composite. Then
\[
  J_n=\ker(\omega_n|_{N_n}),\qquad
  \omega_n(N_n)=\Z.
\]
Consequently, $\omega_n$ induces an isomorphism
\[
  \overline\omega_n:O_n=N_n/J_n\xrightarrow{\sim}\Z.
\]
Moreover, for every proper subgroup $H<S_n$,
\begin{equation}\label{eq:fixed-main}
  \omega_n(\Theta_H)
  =\chi_{\Q[S_n/H]}(c)
  =|(S_n/H)^{\langle c\rangle}|.
\end{equation}
\end{maintheorem}

Following the notation of Section~\ref{subsec:primitive-prelim}, put
\[
  I_n:=I(S_n),\qquad \Prim(S_n)=K(S_n)/I_n.
\]
Since $J_n\subseteq I_n$, there is a natural homomorphism
$\rho_n:O_n\to\Prim(S_n)$.

\begin{maintheorem}\label{thm:primitive-comparison}
Let $n\ge4$.
\begin{enumerate}[label=\textup{(\roman*)},leftmargin=3em]
\item If $n\ge6$ is composite, then $\rho_n:O_n\xrightarrow{\sim}\Prim(S_n)$ is an isomorphism.
\item If $n=4$, then, under the identifications $O_4\cong\Z$ and $\Prim(S_4)\cong\Z/2\Z$, the map $\rho_4$ is reduction modulo $2$.
\item If $n\ge5$ is prime, then $O_n=0$, whereas $\Prim(S_n)\cong\Z$.
\end{enumerate}
\end{maintheorem}

Thus, in particular, for composite degrees $n\ge6$, the restricted induction family consisting only of \emph{proper Young subgroups and standard wreath subgroups} realizes the entire primitive quotient.

In Section~\ref{sec:prelim}, we recall partial Burnside rings, marks, the linearization map, and the Young section. In Section~\ref{sec:standard-lattices}, we construct the standard-relation basis and the Young--wreath lattices. In Section~\ref{sec:ncycle}, we compute the $n$-cycle mark obstruction. In Section~\ref{sec:main-result}, we prove the main theorem, and in Section~\ref{sec:primitive-comparison}, we compare the Young--wreath quotient with the primitive quotient. Finally, Section~\ref{sec:monomial} studies the case $A=C_2$ for the monomial Burnside ring.

\section{Preliminaries}\label{sec:prelim}

\subsection{Notation}

Throughout this paper, all groups are finite and all group actions are left actions.
We write $\operatorname{Sub}(G)$ for the set of all subgroups of $G$.
If $\mathcal D$ is a family of subgroups closed under $G$-conjugation, then
$\mathcal D^{c}$ denotes the set of $G$-conjugacy classes contained in $\mathcal D$.
For a finite $G$-set $X$, we denote its isomorphism class by $[X]$, and for a subgroup
$K\le G$ we write
\[
  X^K=\{x\in X\mid kx=x\text{ for all }k\in K\}
\]
for the set of $K$-fixed points. If $M$ is a subset of an additive group, then
$\langle M\rangle_{\Z}$ denotes the $\Z$-submodule generated by $M$.
Whenever a free $\Z$-module is equipped with a specified basis and $B$ is a
basis element, we write
\[
  \coeff_B(x)
\]
for the coefficient of $B$ in the basis expansion of $x$.

\subsection{Burnside rings and partial Burnside rings}
\label{subsec:partial-burnside}

The isomorphism classes of finite $G$-sets form a commutative semiring, with disjoint
union as addition and Cartesian product as multiplication. The \emph{Burnside ring}
of $G$, denoted by $\Omega(G)$, is the Grothendieck ring obtained by group-completing
this semiring with respect to addition. Thus, for finite $G$-sets $X$ and $Y$,
\[
  [X]+[Y]=[X\sqcup Y],\qquad [X][Y]=[X\times Y],
\]
where $X\times Y$ is equipped with the diagonal $G$-action.
Every finite $G$-set decomposes uniquely as a disjoint union of transitive $G$-sets.
Hence, as an additive group,
\begin{equation}\label{eq:burnside-basis}
  \Omega(G)=
  \bigoplus_{(H)\in\operatorname{Sub}(G)^c}\Z[G/H].
\end{equation}
In particular, the product of two basis elements corresponding to transitive $G$-sets
is given by the double-coset formula
\begin{equation}\label{eq:burnside-product}
  [G/H]\,[G/K]
  =\sum_{HgK\in H\backslash G/K}
  [G/(H\cap gKg^{-1})].
\end{equation}

\begin{definition}\label{def:closed-family}
A family $\mathcal D$ of subgroups of $G$ is called a \emph{closed family of subgroups}
if it satisfies the following two conditions:
\begin{enumerate}[label=\textup{(\roman*)},leftmargin=3em]
\item if $H,K\in\mathcal D$, then $H\cap K\in\mathcal D$;
\item if $H\in\mathcal D$ and $g\in G$, then $gHg^{-1}\in\mathcal D$.
\end{enumerate}
If, in addition, $G\in\mathcal D$, then $\mathcal D$ is called a \emph{collection} of $G$.
\end{definition}

For a closed family $\mathcal D$, put
\begin{equation}\label{eq:partial-burnside}
  \Omega(G,\mathcal D)
  :=\bigoplus_{(H)\in\mathcal D^c}\Z[G/H]
  \subseteq\Omega(G).
\end{equation}
By \eqref{eq:burnside-product} and the closure properties of $\mathcal D$, this is a
subring of $\Omega(G)$. It need not be unital when $G\notin\mathcal D$. If $\mathcal D$
is a collection, then $[G/G]\in\Omega(G,\mathcal D)$, and $\Omega(G,\mathcal D)$ is
called the \emph{partial Burnside ring} relative to $\mathcal D$
\cite{YoshidaGeneralized,IdeiOda,OdaTakegaharaYoshida}.

\subsection{Fixed points, marks, and the Burnside homomorphism}
\label{subsec:marks}

For $K\le G$, the fixed-point count defines a map
\begin{equation}\label{eq:mark-general}
  \phi_K:\Omega(G)\longrightarrow\Z,
  \qquad \phi_K([X])=|X^K|,
\end{equation}
called the $K$-\emph{mark}. Since fixed points commute with disjoint unions and
Cartesian products, $\phi_K$ is a ring homomorphism.

Let $H,K\le G$. A coset $gH\in G/H$ is fixed by $K$ if and only if
\[
  g^{-1}Kg\le H.
\]
Indeed,
\[
\begin{aligned}
  gH\in(G/H)^K
  &\Longleftrightarrow
  kgH=gH
  \quad\text{for every }k\in K\\
  &\Longleftrightarrow
  g^{-1}kg\in H
  \quad\text{for every }k\in K\\
  &\Longleftrightarrow
  g^{-1}Kg\le H.
\end{aligned}
\]
Consequently,
\begin{equation}\label{eq:fixed-coset-formula}
  \phi_K([G/H])
  =
  |(G/H)^K|
  =
  \frac{1}{|H|}
  \left|
    \left\{
      g\in G
      \mathrel{}\middle|\mathrel{}
      g^{-1}Kg\le H
    \right\}
  \right|.
\end{equation}
The set on the right-hand side is stable under right multiplication by $H$,
and its right $H$-orbits are in bijection with the cosets in $(G/H)^K$.

For a closed family $\mathcal D$, the restricted Burnside homomorphism
\begin{equation}\label{eq:partial-burnside-hom}
  \Phi_{\mathcal D}:\Omega(G,\mathcal D)
  \longrightarrow\prod_{(K)\in\mathcal D^c}\Z,
  \qquad x\longmapsto(\phi_K(x))_{(K)\in\mathcal D^c},
\end{equation}
is injective \cite[Lemma~3.3]{YoshidaGeneralized}. The matrix
\[
  M_{\mathcal D}=
  \bigl(|(G/H)^K|\bigr)_{(K),(H)\in\mathcal D^c}
\]
is called the \emph{table of marks} relative to $\mathcal D$. The $C_n$-mark used in
this paper is the special case of \eqref{eq:mark-general} with $K=C_n$.

\subsection{Biset functors and the linearization map}
\label{subsec:burnside-functoriality}

Let $\mathbf{Ab}$ denote the category of abelian groups and group homomorphisms.

Let $\mathcal B$ be the biset category of finite groups. Its objects are finite groups.
For finite groups $G$ and $H$, the abelian group of morphisms from $G$ to $H$ is
\[
  \Hom_{\mathcal B}(G,H),
\]
defined as the Grothendieck group, with respect to disjoint union, of the commutative
monoid of isomorphism classes of finite $(H,G)$-bisets.

If $U$ is a finite $(H,G)$-biset and $V$ is a finite $(K,H)$-biset, then the
composition of the corresponding morphisms is represented by the balanced product
\[
  V\times_H U.
\]
Here
\[
  V\times_H U
  :=(V\times U)/{\sim},
  \qquad
  (vh,u)\sim(v,hu)
  \quad
  (v\in V,\ u\in U,\ h\in H).
\]
The identity morphism of a finite group $G$ is represented by the regular
$(G,G)$-biset ${}_GG_G$, with the left and right multiplication actions.

An additive functor
\[
  F:\mathcal B\longrightarrow\mathbf{Ab}
\]
is called a \emph{biset functor}. Thus, it assigns an abelian group $F(G)$ to every
finite group $G$ and a group homomorphism
\[
  F(U):F(G)\longrightarrow F(H)
\]
to every finite $(H,G)$-biset $U$, preserving identity morphisms and composition by
balanced products, and carrying disjoint unions of bisets to sums of homomorphisms.

The Burnside rings define a biset functor
\[
  \Omega:\mathcal B\longrightarrow\mathbf{Ab}.
\]
For a finite $(H,G)$-biset $U$ and $x\in\Omega(G)$, we denote the element induced by
$U$ by
\[
  U\cdot x\in\Omega(H).
\]
In particular, for a finite $G$-set $X$,
\[
  U\cdot[X]=[U\times_G X].
\]

For a finite group $G$, let $R_{\Q}(G)$ denote the Grothendieck ring of
finite-dimensional $\Q G$-modules. These rings likewise define a biset functor
\[
  R_{\Q}:\mathcal B\longrightarrow\mathbf{Ab}.
\]
For a finite $(H,G)$-biset $U$ and a finite-dimensional $\Q G$-module $V$, we write
\[
  U\cdot[V]
  :=[\Q U\otimes_{\Q G}V]
  \in R_{\Q}(H),
\]
where $\Q U$ is the $(\Q H,\Q G)$-bimodule with basis $U$.

For each finite group $G$, the ring homomorphism
\begin{equation}\label{eq:linearization-general}
  \ell_G:\Omega(G)\longrightarrow R_{\Q}(G),
  \qquad [X]\longmapsto[\Q[X]],
\end{equation}
is called the \emph{linearization map}. The family $(\ell_G)_G$ defines a natural
transformation of biset functors
\[
  \ell:\Omega\Longrightarrow R_{\Q}.
\]
Equivalently, for every finite $(H,G)$-biset $U$ and every $x\in\Omega(G)$,
\begin{equation}\label{eq:linearization-naturality}
  \ell_H(U\cdot x)=U\cdot\ell_G(x).
\end{equation}
This follows from the natural isomorphism of $\Q H$-modules
\[
  \Q[U\times_G X]
  \cong
  \Q U\otimes_{\Q G}\Q[X]
\]
for every finite $G$-set $X$.

We next describe induction, restriction, and inflation as special cases of this biset
functor structure. Let $L\le G$. The $(L,G)$-biset ${}_LG_G$ and the $(G,L)$-biset
${}_GG_L$ induce, respectively,
\[
  \Res_L^G:\Omega(G)\longrightarrow\Omega(L),
  \qquad
  \Ind_L^G:\Omega(L)\longrightarrow\Omega(G).
\]
On transitive sets, these maps are given by
\begin{align}
  \Res_L^G[G/H]
  &=\sum_{LgH\in L\backslash G/H}
  [L/(L\cap gHg^{-1})],
  \label{eq:restriction}\\
  \Ind_L^G[L/H]
  &=[G/H].
  \label{eq:induction}
\end{align}

If $N\triangleleft G$, then the $(G,G/N)$-biset ${}_G(G/N)_{G/N}$ obtained from the
quotient map $G\to G/N$ induces inflation
\[
  \Inf_{G/N}^{G}:\Omega(G/N)\longrightarrow\Omega(G).
\]
For $N\le H\le G$,
\[
  \Inf_{G/N}^{G}\bigl([(G/N)/(H/N)]\bigr)=[G/H].
\]

By the naturality in \eqref{eq:linearization-naturality}, these maps commute with the
corresponding maps on representation rings. Explicitly,
\[
\begin{aligned}
  \ell_L\circ\Res_L^G
  &=\Res_L^G\circ\ell_G,\\
  \ell_G\circ\Ind_L^G
  &=\Ind_L^G\circ\ell_L,\\
  \ell_G\circ\Inf_{G/N}^G
  &=\Inf_{G/N}^G\circ\ell_{G/N}.
\end{aligned}
\]
Moreover, Frobenius reciprocity for the multiplicative structure of the Burnside
biset functor gives the projection formula
\begin{equation}\label{eq:projection-formula}
  \Ind_L^G\bigl(x\,\Res_L^G(y)\bigr)
  =\Ind_L^G(x)\,y
  \qquad
  (x\in\Omega(L),\ y\in\Omega(G)).
\end{equation}
For the general theory of biset categories and biset functors, see
Bouc~\cite{BoucBiset}.

\subsection{Brauer relations and the primitive quotient}
\label{subsec:primitive-prelim}

For a finite group $G$, put
\[
  K(G):=\ker\ell_G.
\]
An element of $K(G)$ is called a \emph{Brauer relation}, or a \emph{$G$-relation}.

By the naturality of the linearization map, for every finite $(H,G)$-biset $U$ one has
\[
  U\cdot K(G)\subseteq K(H).
\]
Consequently, the family $K=(K(G))_G$ is a biset subfunctor of the Burnside biset
functor $\Omega$. In particular, restriction, induction, and inflation send Brauer
relations to Brauer relations.

Define
\begin{equation}\label{eq:imprimitive-lattice-general}
  I(G):=
  \sum_{H<G}\Ind_H^G K(H)
  +
  \sum_{\substack{N\triangleleft G\\N\ne1}}
  \Inf_{G/N}^{G}K(G/N)
  \subseteq K(G).
\end{equation}
Thus $I(G)$ is the subgroup generated by Brauer relations induced from proper
subgroups and by Brauer relations inflated from proper quotient groups. Its elements
are called \emph{imprimitive relations}, and the quotient abelian group
\begin{equation}\label{eq:primitive-quotient-general}
  \Prim(G):=K(G)/I(G)
\end{equation}
is called the \emph{primitive Brauer quotient}.

A Brauer relation $\Theta\in K(G)$ is called a \emph{primitive relation} if its class
\[
  \Theta+I(G)\in\Prim(G)
\]
is nonzero. Thus $\Prim(G)$ measures the Brauer relations that cannot be generated by
induction from proper subgroups and inflation from proper quotient groups. For these
definitions and the classification of primitive Brauer quotients for arbitrary finite
groups, see Bartel--Dokchitser~\cite{BartelDokchitserBrauer}.

\subsection{Young partial Burnside rings and the Young section}
\label{subsec:young-partial}

For a partition $\lambda=(\lambda_1,\ldots,\lambda_r)\vdash n$, let
\[
  S_\lambda=S_{\lambda_1}\times\cdots\times S_{\lambda_r}\le S_n
\]
be the corresponding Young subgroup. Let $\mathcal Y_n$ denote the family consisting
of the Young subgroups and their conjugates. The intersection of two conjugates of
Young subgroups is the stabilizer of the common refinement of the corresponding set
partitions. Hence $\mathcal Y_n$ is a collection of $S_n$, and
\begin{equation}\label{eq:young-partial-ring}
  P_n:=\Omega(S_n,\mathcal Y_n)
  =\bigoplus_{\lambda\vdash n}\Z[S_n/S_\lambda]
\end{equation}
is a partial Burnside ring.

Since every irreducible character of $S_n$ can be realized over $\Q$, we identify
$R_{\Q}(S_n)$ with the ordinary character ring \cite{JamesKerber}. Young permutation
characters form a $\Z$-basis of $R_{\Q}(S_n)$ by Young's rule and the integral
unitriangularity of the Kostka matrix
\cite{JamesKerber,Macdonald,IdeiOda}. Therefore, the restriction of the linearization
map
\begin{equation}\label{eq:young-linearization}
  \ell_n:=\ell_{S_n}|_{P_n}:P_n\xrightarrow{\sim}R_{\Q}(S_n)
\end{equation}
is an isomorphism of rings. We denote its inverse by
\begin{equation}\label{eq:young-section}
  \sigma_n:R_{\Q}(S_n)\xrightarrow{\sim}P_n
\end{equation}
and call it the \emph{Young section}.

More generally, let
$L=S_{\lambda_1}\times\cdots\times S_{\lambda_r}$, and let
$P_L\subseteq\Omega(L)$ be the subring spanned by the external products of coset sets
of Young subgroups in the direct factors of $L$. The same argument shows that
\[
  \ell_L|_{P_L}:P_L\xrightarrow{\sim}R_{\Q}(L).
\]
We denote its inverse by $\sigma_L$.

\section{Standard Brauer relations and the Young--wreath lattice}\label{sec:standard-lattices}

\subsection{A basis of standard Brauer relations}

Using the Young section $\sigma_n$ introduced in Section~\ref{subsec:young-partial},
we consider the standard relations $\Theta_H$ defined in
\eqref{eq:theta-intro}.

\begin{proposition}\label{prop:standard-basis}
As $H$ ranges over representatives of the $S_n$-conjugacy classes of subgroups
that are not conjugate to Young subgroups, the elements $\Theta_H$ form a
$\Z$-basis of $K(S_n)$.
\end{proposition}

\begin{proof}
Let $F_n$ be the free $\Z$-submodule generated by the Burnside basis elements
that do not correspond to Young coset spaces. Then
\[
  \Omega(S_n)=P_n\oplus F_n.
\]

The composite
\[
  \sigma_n\circ\ell_{S_n}:
  \Omega(S_n)\longrightarrow P_n
\]
is a $\Z$-module homomorphism. Hence the map
\[
  \pi_n:\Omega(S_n)\longrightarrow\Omega(S_n),
  \qquad
  \pi_n(x)
  :=
  x-\sigma_n\bigl(\ell_{S_n}(x)\bigr)
\]
is a $\Z$-module homomorphism.

Moreover,
\[
  \ell_{S_n}\bigl(\pi_n(x)\bigr)=0,
\]
so the image of $\pi_n$ is contained in $K(S_n)$. We therefore regard $\pi_n$
as a map
\[
  \pi_n:\Omega(S_n)\longrightarrow K(S_n).
\]

If $k\in K(S_n)$, then
\[
  \pi_n(k)=k,
\]
whereas $\pi_n(p)=0$ for every $p\in P_n$. Conversely, if $\pi_n(x)=0$, then
$x=\sigma_n(\ell_{S_n}(x))\in P_n$. Thus $\pi_n$ is a projection onto
$K(S_n)$ with kernel $P_n$.

We first show that $\pi_n|_{F_n}$ is injective. Suppose that $z\in F_n$ and
$\pi_n(z)=0$. Then
\[
  z=\sigma_n\bigl(\ell_{S_n}(z)\bigr)\in P_n.
\]
Consequently,
\[
  z\in P_n\cap F_n=0,
\]
and hence $z=0$.

We next show that $\pi_n|_{F_n}$ is surjective. Let
$k\in K(S_n)\subseteq\Omega(S_n)$ and write
\[
  k=p+z
  \qquad
  (p\in P_n,\ z\in F_n).
\]
Since $\ell_{S_n}(k)=0$, we have
\[
  \ell_{S_n}(p)=-\ell_{S_n}(z).
\]
Since $\sigma_n$ is the inverse of $\ell_{S_n}|_{P_n}$, it follows that
\[
  p=-\sigma_n\bigl(\ell_{S_n}(z)\bigr).
\]
Therefore,
\[
  k=z-\sigma_n\bigl(\ell_{S_n}(z)\bigr)=\pi_n(z).
\]
Thus
\[
  \pi_n|_{F_n}:F_n\xrightarrow{\sim}K(S_n)
\]
is an isomorphism of $\Z$-modules. For every subgroup $H\le S_n$ that is not
conjugate to a Young subgroup,
\[
  \pi_n([S_n/H])=\Theta_H.
\]
Hence the images of the standard basis of $F_n$ are precisely the elements
$\Theta_H$, and these elements form a $\Z$-basis of $K(S_n)$.
\end{proof}

\subsection{Young-induced relations and intransitive subgroups}
\label{subsec:young-induced-relations}

Define the $\Z$-submodule generated by the Brauer relations induced from proper
Young subgroups by
\begin{equation}\label{eq:young-induced-lattice}
  J_n^{\mathrm Y}
  :=
  \sum_{\substack{L<S_n\\ L\text{ is a Young subgroup}}}
  \Ind_L^{S_n}K(L)
  \subseteq K(S_n).
\end{equation}
The final inclusion follows from the naturality of the linearization map.

\begin{definition}\label{def:intransitive-subgroup}
A subgroup $H\le S_n$ is called \emph{intransitive} if its natural action on
$\{1,\ldots,n\}$ is intransitive.
\end{definition}

Let $H\le S_n$ be intransitive, and write its orbit decomposition as
\[
  \{1,\ldots,n\}
  =
  \mathcal O_1\sqcup\cdots\sqcup\mathcal O_r
  \qquad
  (r\ge2).
\]
Then
\[
  L_H
  :=
  S_{\mathcal O_1}\times\cdots\times S_{\mathcal O_r}
  \le S_n
\]
is a proper Young subgroup and $H\le L_H$.

Now let $L<S_n$ be a proper Young subgroup. First,
\[
  \Ind_L^{S_n}(P_L)\subseteq P_n.
\]
Indeed, the basis elements of $P_L$ are external products of coset spaces of
Young subgroups in the direct factors of $L$, and their inductions to $S_n$ are
Young coset spaces of $S_n$.

By the naturality of the linearization map,
\[
  \ell_n\circ
  \left(\Ind_L^{S_n}|_{P_L}\right)
  =
  \Ind_L^{S_n}\circ
  \left(\ell_L|_{P_L}\right).
\]
Here $\ell_L|_{P_L}$ and $\ell_n$ are isomorphisms, with inverses $\sigma_L$
and $\sigma_n$, respectively. Therefore,
\begin{equation}\label{eq:section-compatible}
  \Ind_L^{S_n}\circ\sigma_L
  =
  \sigma_n\circ\Ind_L^{S_n}.
\end{equation}

For such an $L$, define
\[
  \pi_L:\Omega(L)\longrightarrow K(L),
  \qquad
  \pi_L(x)
  :=
  x-\sigma_L\bigl(\ell_L(x)\bigr).
\]
Then $\pi_L(x)\in K(L)$.

Equation~\eqref{eq:section-compatible} and the naturality of the linearization
map imply that
\begin{equation}\label{eq:pi-induction-compatible}
  \Ind_L^{S_n}\circ\pi_L
  =
  \pi_n\circ\Ind_L^{S_n}.
\end{equation}
Indeed, for $x\in\Omega(L)$,
\[
\begin{aligned}
  \Ind_L^{S_n}\bigl(\pi_L(x)\bigr)
  &=
  \Ind_L^{S_n}(x)
  -
  \Ind_L^{S_n}
  \sigma_L\bigl(\ell_L(x)\bigr)\\
  &=
  \Ind_L^{S_n}(x)
  -
  \sigma_n
  \Ind_L^{S_n}\bigl(\ell_L(x)\bigr)\\
  &=
  \Ind_L^{S_n}(x)
  -
  \sigma_n
  \ell_{S_n}\bigl(\Ind_L^{S_n}(x)\bigr)\\
  &=
  \pi_n\bigl(\Ind_L^{S_n}(x)\bigr).
\end{aligned}
\]

\begin{proposition}\label{prop:intransitive-in-Y}
For every intransitive subgroup $H\le S_n$,
\[
  \Theta_H\in J_n^{\mathrm Y}.
\]
\end{proposition}

\begin{proof}
Let $L_H<S_n$ be the proper Young subgroup corresponding to the orbit
decomposition of $H$. Then $H\le L_H$ and
\[
  [S_n/H]
  =
  \Ind_{L_H}^{S_n}[L_H/H].
\]
Hence, by \eqref{eq:pi-induction-compatible},
\[
\begin{aligned}
  \Theta_H
  &=
  \pi_n([S_n/H])\\
  &=
  \pi_n\left(
    \Ind_{L_H}^{S_n}[L_H/H]
  \right)\\
  &=
  \Ind_{L_H}^{S_n}
  \left(
    \pi_{L_H}([L_H/H])
  \right).
\end{aligned}
\]
Since
\[
  \pi_{L_H}([L_H/H])\in K(L_H),
\]
we obtain
\[
  \Theta_H
  \in
  \Ind_{L_H}^{S_n}K(L_H)
  \subseteq
  J_n^{\mathrm Y}.
\]
\end{proof}

\subsection{Action-imprimitive subgroups and standard wreath subgroups}
\label{subsec:action-imprimitive}

\begin{definition}\label{def:block-system}
Let $H\le S_n$. A partition
\[
  \mathcal B=\{B_1,\ldots,B_b\}
\]
of $\{1,\ldots,n\}$ is called a \emph{block system} for $H$ if $H$ permutes
the blocks of $\mathcal B$, that is,
\[
  h\mathcal B=\mathcal B
  \qquad
  (h\in H),
\]
where
\[
  h\mathcal B
  :=
  \{hB_1,\ldots,hB_b\}.
\]

The block systems
\[
  \bigl\{\{1,\ldots,n\}\bigr\}
  \qquad\text{and}\qquad
  \bigl\{\{1\},\ldots,\{n\}\bigr\}
\]
are called \emph{trivial}; every other block system is called
\emph{nontrivial}.

The stabilizer of $\mathcal B$ in $S_n$ is defined by
\[
  \operatorname{Stab}_{S_n}(\mathcal B)
  :=
  \{g\in S_n\mid g\mathcal B=\mathcal B\}.
\]
\end{definition}

\begin{definition}\label{def:action-imprimitive-subgroup}
A subgroup $H\le S_n$ is called \emph{transitive imprimitive} if it acts
transitively on $\{1,\ldots,n\}$ and preserves a nontrivial block system.

For brevity, a subgroup $H\le S_n$ will be called \emph{action-imprimitive}
if it is either intransitive or transitive imprimitive. This term is used only
as an abbreviation within the present paper. We denote the set of all such
subgroups by $\mathcal H_n^{\mathrm{np}}$.
\end{definition}

By Proposition~\ref{prop:standard-basis}, the standard relations associated
with action-imprimitive subgroups that are not conjugate to Young subgroups
are linearly independent. Accordingly, the Young--wreath lattice is
\begin{equation}\label{eq:nonprimitive-lattice}
  N_n
  :=
  \bigoplus_{\substack{
    (H)\in\operatorname{Sub}(S_n)^c\\
    H\in\mathcal H_n^{\mathrm{np}}\\
    H\text{ is not conjugate to a Young subgroup}
  }}
  \Z\Theta_H
  \subseteq K(S_n).
\end{equation}
Equivalently, the displayed elements $\Theta_H$ form a $\Z$-basis of $N_n$.

Let $n=ab$ with $a,b>1$. Fix a partition
\[
  \mathcal B_{a,b}
  =
  \{B_1,\ldots,B_b\},
  \qquad
  |B_i|=a
  \quad
  (1\le i\le b),
\]
of $\{1,\ldots,n\}$. Then $\mathcal B_{a,b}$ is a nontrivial block system.
We call its stabilizer
\[
  W_{a,b}
  :=
  \operatorname{Stab}_{S_n}(\mathcal B_{a,b})
\]
the \emph{standard wreath subgroup}. We have
\[
  W_{a,b}\cong S_a\wr S_b.
\]
Any two such partitions are conjugate under $S_n$, so the conjugacy class of
$W_{a,b}$ is independent of the chosen partition.

Suppose that a transitive imprimitive subgroup $H\le S_n$ preserves a block
system with blocks of size $a$ and with $b$ blocks. Then $n=ab$ with $a,b>1$,
and there exists $g\in S_n$ such that
\[
  gHg^{-1}\le W_{a,b}.
\]

In addition to the relations induced from proper Young subgroups, define the
$\Z$-submodule generated by the Brauer relations induced from standard wreath
subgroups by
\[
  J_n^{\mathrm W}
  :=
  \sum_{\substack{ab=n\\a,b>1}}
  \Ind_{W_{a,b}}^{S_n}K(W_{a,b})
  \subseteq K(S_n),
\]
and set
\begin{equation}\label{eq:young-wreath-induced-lattice}
  J_n
  :=
  J_n^{\mathrm Y}+J_n^{\mathrm W}
  \subseteq K(S_n).
\end{equation}

\begin{lemma}\label{lem:all-nonprimitive-induction}
If $H\in\mathcal H_n^{\mathrm{np}}$, then
\[
  \Ind_H^{S_n}K(H)\subseteq J_n.
\]
\end{lemma}

\begin{proof}
Suppose first that $H$ is intransitive. Let $L<S_n$ be the proper Young
subgroup corresponding to the orbit decomposition of $H$. Then $H\le L$.
By the naturality of the linearization map,
\[
  \Ind_H^L K(H)\subseteq K(L).
\]
The transitivity of induction therefore gives
\[
\begin{aligned}
  \Ind_H^{S_n}K(H)
  &=
  \Ind_L^{S_n}\bigl(\Ind_H^L K(H)\bigr)\\
  &\subseteq
  \Ind_L^{S_n}K(L)
  \subseteq
  J_n^{\mathrm Y}
  \subseteq
  J_n.
\end{aligned}
\]

Suppose next that $H$ is transitive imprimitive. After replacing $H$ by a
conjugate, we may assume that $H\le W_{a,b}$ for some factorization $n=ab$ with
$a,b>1$. Similarly,
\[
  \Ind_H^{W_{a,b}}K(H)\subseteq K(W_{a,b}),
\]
and hence
\[
\begin{aligned}
  \Ind_H^{S_n}K(H)
  &=
  \Ind_{W_{a,b}}^{S_n}
  \bigl(\Ind_H^{W_{a,b}}K(H)\bigr)\\
  &\subseteq
  \Ind_{W_{a,b}}^{S_n}K(W_{a,b})
  \subseteq
  J_n^{\mathrm W}
  \subseteq
  J_n.
\end{aligned}
\]
\end{proof}

\begin{lemma}\label{lem:J-in-N}
We have $J_n\subseteq N_n$.
\end{lemma}

\begin{proof}
Let $L$ be either a proper Young subgroup or a standard wreath subgroup
$W_{a,b}$, and let $y\in\Ind_L^{S_n}K(L)$. There exists $z\in K(L)$ such that
\[
  y=\Ind_L^{S_n}(z).
\]
Expand $z$ in the Burnside basis of $\Omega(L)$ as
\[
  z
  =
  \sum_{(H)\in\operatorname{Sub}(L)^c}
  a_H[L/H]
  \qquad
  (a_H\in\Z).
\]
By the definition of induction,
\[
\begin{aligned}
  y
  &=
  \Ind_L^{S_n}(z)\\
  &=
  \sum_{(H)\in\operatorname{Sub}(L)^c}
  a_H\Ind_L^{S_n}[L/H]\\
  &=
  \sum_{(H)\in\operatorname{Sub}(L)^c}
  a_H[S_n/H].
\end{aligned}
\]
If distinct $L$-conjugacy classes become conjugate in $S_n$, we combine the
corresponding terms. It follows that, for every Burnside basis element
$[S_n/K]$ occurring with nonzero coefficient in $y$, the subgroup $K$ is
$S_n$-conjugate to a subgroup of $L$.

If $L$ is a proper Young subgroup, then $L$ fixes setwise every part of a
nontrivial partition of $\{1,\ldots,n\}$. Hence every subgroup of $L$ is
intransitive. If $L=W_{a,b}$, every subgroup of $L$ preserves the nontrivial
block system $\mathcal B_{a,b}$ and is therefore either intransitive or
transitive imprimitive. Thus every subgroup $K$ corresponding to a Burnside
basis element occurring in $y$ is action-imprimitive.

On the other hand, Proposition~\ref{prop:standard-basis} gives the unique
standard-relation expansion
\[
  y
  =
  \sum_{\substack{(H)\in\operatorname{Sub}(S_n)^c\\
          H\text{ is not a Young subgroup}}}
  m_H\Theta_H.
\]
For every non-Young subgroup $H$, write
\[
  \Theta_H=[S_n/H]+p_H
  \qquad
  (p_H\in P_n).
\]
Since $p_H$ is a linear combination of Young basis elements, it does not
contribute to the coefficient of any non-Young Burnside basis element.
Consequently, $m_H$ is exactly the coefficient of $[S_n/H]$ in the Burnside
basis expansion of $y$.

It follows that $m_H=0$ whenever the natural action of $H$ is primitive, and
therefore $y\in N_n$. Applying this argument to every summand defining $J_n$
gives
\[
  J_n\subseteq N_n.
\]
\end{proof}

\section{The \texorpdfstring{$n$}{n}-cycle mark obstruction}\label{sec:ncycle}

\subsection{The fixed-coset formula}
\label{subsec:fixed-coset-formula}

Put
\[
  c=(1\ 2\ \cdots\ n)\in S_n,
  \qquad
  C_n:=\langle c\rangle.
\]

Define
\begin{equation}\label{eq:omega-definition}
  \omega_n:K(S_n)\longrightarrow\Z,
  \qquad
  \omega_n(x):=-\coeff_{[S_n/S_n]}(x).
\end{equation}
We call this $\Z$-module homomorphism the
\emph{$n$-cycle mark obstruction}.

Applying \eqref{eq:fixed-coset-formula} with $G=S_n$ and $K=C_n$, we obtain
\[
  \phi_{C_n}([S_n/H])
  =
  |(S_n/H)^{C_n}|
  =
  \frac{1}{|H|}
  \left|
    \left\{
      g\in S_n
      \mathrel{}\middle|\mathrel{}
      g^{-1}C_ng\le H
    \right\}
  \right|
\]
for every subgroup $H\le S_n$.

Moreover, the value of a permutation character is the number of fixed points
of the corresponding permutation action. Since $C_n=\langle c\rangle$, we have
\begin{equation}\label{eq:permutation-character-fixed-points}
  \chi_{\Q[S_n/H]}(c)
  =
  |(S_n/H)^{\langle c\rangle}|
  =
  |(S_n/H)^{C_n}|
  =
  \phi_{C_n}([S_n/H]).
\end{equation}

\begin{proposition}\label{prop:fixed-coset}
Let $H\le S_n$ be a subgroup that is not conjugate to a Young subgroup. Then
\[
  \omega_n(\Theta_H)
  =
  \phi_{C_n}([S_n/H])
  =
  |(S_n/H)^{C_n}|.
\]
\end{proposition}

\begin{proof}
Write the expansion with respect to the Young section as
\[
  \sigma_n\bigl(\ell_{S_n}([S_n/H])\bigr)
  =
  \sum_{\lambda\vdash n}
  a_{H,\lambda}[S_n/S_\lambda].
\]

If $\lambda\ne(n)$, then $S_\lambda$ is a proper Young subgroup and its
natural action is intransitive. On the other hand, $C_n$ acts transitively on
$\{1,\ldots,n\}$. Hence no conjugate of $C_n$ is contained in $S_\lambda$, and
\eqref{eq:fixed-coset-formula} gives
\[
  \phi_{C_n}([S_n/S_\lambda])=0
  \qquad
  (\lambda\ne(n)).
\]
Since $S_{(n)}=S_n$, we also have
\[
  \phi_{C_n}([S_n/S_{(n)}])
  =
  \phi_{C_n}([S_n/S_n])
  =
  1.
\]

After applying the linearization map and evaluating the resulting character at
$c$, we obtain
\[
\begin{aligned}
  |(S_n/H)^{C_n}|
  &=
  \chi_{\Q[S_n/H]}(c)\\
  &=
  \sum_{\lambda\vdash n}
  a_{H,\lambda}
  \chi_{\Q[S_n/S_\lambda]}(c)\\
  &=
  a_{H,(n)}.
\end{aligned}
\]

On the other hand,
\[
\begin{aligned}
  \Theta_H
  &=
  [S_n/H]
  -
  \sigma_n\bigl(\ell_{S_n}([S_n/H])\bigr)\\
  &=
  [S_n/H]
  -
  \sum_{\lambda\vdash n}
  a_{H,\lambda}[S_n/S_\lambda].
\end{aligned}
\]
Since $H$ is not conjugate to a Young subgroup,
\[
  \coeff_{[S_n/S_n]}(\Theta_H)=-a_{H,(n)}.
\]
Hence, by \eqref{eq:omega-definition},
\[
  \omega_n(\Theta_H)
  =
  a_{H,(n)}
  =
  |(S_n/H)^{C_n}|
  =
  \phi_{C_n}([S_n/H]).
\]
\end{proof}

\subsection{Vanishing on induced relations and surjectivity}

\begin{lemma}\label{lem:omega-J}
We have $J_n\subseteq\ker\omega_n$.
\end{lemma}

\begin{proof}
Let $L$ be either a proper Young subgroup or a standard wreath subgroup
$W_{a,b}$, and let
\[
  y\in\Ind_L^{S_n}K(L).
\]
There exists $z\in K(L)$ such that
\[
  y=\Ind_L^{S_n}(z).
\]

Expand $z$ in the Burnside basis of $\Omega(L)$ as
\[
  z
  =
  \sum_{(H)\in\operatorname{Sub}(L)^c}
  a_H[L/H]
  \qquad
  (a_H\in\Z).
\]
By the definition of induction,
\[
  y
  =
  \sum_{(H)\in\operatorname{Sub}(L)^c}
  a_H[S_n/H].
\]
Since $H\le L<S_n$, every such $H$ is a proper subgroup of $S_n$. Even after
combining terms whose $L$-conjugacy classes become conjugate in $S_n$, the
basis element $[S_n/S_n]$ does not occur in this expansion. Hence
\[
  \coeff_{[S_n/S_n]}(y)=0,
  \qquad
  \omega_n(y)=0.
\]
Applying this argument to every generating submodule of $J_n^{\mathrm Y}$ and
$J_n^{\mathrm W}$ gives
\[
  J_n\subseteq\ker\omega_n.
\]
\end{proof}

\begin{lemma}\label{lem:wreath-one}
Let $n=ab$ with $a,b>1$. Then
\[
  \omega_n(\Theta_{W_{a,b}})
  =
  |(S_n/W_{a,b})^{C_n}|
  =
  1.
\]
\end{lemma}

\begin{proof}
The coset space $S_n/W_{a,b}$ can be identified with the set of all block
systems on $\{1,\ldots,n\}$ consisting of $b$ blocks of size $a$. The unique
block system of this type preserved by the regular cyclic group $C_n$ is the
partition into the cosets of the unique subgroup of $C_n$ of order $a$.
Consequently, $|(S_n/W_{a,b})^{C_n}|=1$, and the result follows from
Proposition~\ref{prop:fixed-coset}.
\end{proof}

\begin{corollary}\label{cor:surjective}
If $n\ge4$ is composite, then
\[
  \omega_n|_{N_n}:N_n\longrightarrow\Z
\]
is surjective.
\end{corollary}

\begin{proof}
Choose a factorization $n=ab$ with $a,b>1$. Since $W_{a,b}$ is transitive
imprimitive, we have $\Theta_{W_{a,b}}\in N_n$. By
Lemma~\ref{lem:wreath-one},
\[
  \omega_n(\Theta_{W_{a,b}})=1.
\]
Hence the image of $\omega_n|_{N_n}$ contains $1$, and the map is surjective.
\end{proof}

\section{The Young--wreath quotient}\label{sec:main-result}

\begin{lemma}\label{lem:projection}
Let $H\in\mathcal H_n^{\mathrm{np}}$ and $x\in K(S_n)$. Then
\[
  [S_n/H]x\in J_n.
\]
\end{lemma}

\begin{proof}
Apply the projection formula \eqref{eq:projection-formula} with $L=H$, with the
unit $[H/H]\in\Omega(H)$ as its first input, and with
$x\in K(S_n)\subseteq\Omega(S_n)$ as its second input. We obtain
\[
\begin{aligned}
  \Ind_H^{S_n}\Res_H^{S_n}(x)
  &=
  \Ind_H^{S_n}\bigl([H/H]\,\Res_H^{S_n}(x)\bigr)\\
  &=
  \Ind_H^{S_n}([H/H])\,x\\
  &=
  [S_n/H]x.
\end{aligned}
\]
Thus
\[
  [S_n/H]x
  =
  \Ind_H^{S_n}\Res_H^{S_n}(x).
\]
Since the linearization map commutes with restriction,
\[
  \Res_H^{S_n}(x)\in K(H).
\]
The right-hand side therefore belongs to $J_n$ by
Lemma~\ref{lem:all-nonprimitive-induction}.
\end{proof}

\begin{lemma}\label{lem:zero-support}
If $x\in N_n$ and $\omega_n(x)=0$, then
\[
  x\in
  \left\langle
    [S_n/H]
    \mathrel{}\middle|\mathrel{}
    H\in\mathcal H_n^{\mathrm{np}}
  \right\rangle_{\Z}.
\]
\end{lemma}

\begin{proof}
By \eqref{eq:nonprimitive-lattice}, $x$ has a unique expansion
\[
  x
  =
  \sum_{\substack{
    (H)\in\operatorname{Sub}(S_n)^c\\
    H\in\mathcal H_n^{\mathrm{np}}\\
    H\text{ is not conjugate to a Young subgroup}
  }}
  m_H\Theta_H
  \qquad
  (m_H\in\Z).
\]
For every subgroup occurring in this sum,
\[
  \Theta_H
  =
  [S_n/H]
  -
  \sigma_n\bigl(\ell_{S_n}([S_n/H])\bigr).
\]
The Young-section term is a $\Z$-linear combination of Young coset spaces.
Every proper Young subgroup is intransitive and hence belongs to
$\mathcal H_n^{\mathrm{np}}$. Consequently, the only Burnside basis element
that may occur in $x$ without being of the form $[S_n/H]$ for some
$H\in\mathcal H_n^{\mathrm{np}}$ is $[S_n/S_n]$. By the definition of
$\omega_n$,
\[
  \coeff_{[S_n/S_n]}(x)
  =
  -\omega_n(x)
  =
  0.
\]
The assertion follows.
\end{proof}

\begin{lemma}\label{lem:wreath-unit}
Let $n=ab$ with $a,b>1$, and set
\[
\begin{aligned}
  \mathcal R_{a,b}
  &:=-\Theta_{W_{a,b}}\\
  &=-[S_n/W_{a,b}]
  +\sigma_n\bigl(\ell_{S_n}([S_n/W_{a,b}])\bigr).
\end{aligned}
\]
Then there exists
\[
  u\in
  \left\langle
    [S_n/H]
    \mathrel{}\middle|\mathrel{}
    H\in\mathcal H_n^{\mathrm{np}}
  \right\rangle_{\Z}
\]
such that
\[
  \mathcal R_{a,b}=[S_n/S_n]+u.
\]
\end{lemma}

\begin{proof}
By Lemma~\ref{lem:wreath-one},
\[
  \omega_n(\Theta_{W_{a,b}})=1.
\]
By the definition of $\omega_n$,
\[
  \coeff_{[S_n/S_n]}(\mathcal R_{a,b})
  =
  -\coeff_{[S_n/S_n]}(\Theta_{W_{a,b}})
  =
  \omega_n(\Theta_{W_{a,b}})
  =
  1.
\]

Moreover, $W_{a,b}$ is transitive imprimitive, so
$W_{a,b}\in\mathcal H_n^{\mathrm{np}}$. The Young-section term in
$\mathcal R_{a,b}$ is a $\Z$-linear combination of the coset spaces
$[S_n/S_\lambda]$. For $\lambda\ne(n)$, the subgroup $S_\lambda$ is a proper
Young subgroup and is therefore intransitive. Thus, after separating the term
$[S_n/S_{(n)}]=[S_n/S_n]$, every remaining Burnside basis element is of the
form $[S_n/H]$ with $H\in\mathcal H_n^{\mathrm{np}}$. This proves the assertion.
\end{proof}

\begin{maintheoremrestated}
Let $n\ge4$ be composite. Then
\[
  J_n=\ker\bigl(\omega_n|_{N_n}\bigr),
  \qquad
  \omega_n(N_n)=\Z.
\]
Consequently, $\omega_n$ induces an isomorphism
\[
  \overline\omega_n:
  O_n=N_n/J_n\xrightarrow{\sim}\Z.
\]
Moreover, for every proper subgroup $H<S_n$,
\[
  \omega_n(\Theta_H)
  =
  \chi_{\Q[S_n/H]}(c)
  =
  |(S_n/H)^{\langle c\rangle}|.
\]
\end{maintheoremrestated}

\begin{proof}
We first prove the final assertion. If $H$ is not conjugate to a Young subgroup,
it follows from Proposition~\ref{prop:fixed-coset} and
\eqref{eq:permutation-character-fixed-points}. If $H$ is conjugate to a proper
Young subgroup, then $\Theta_H=0$. Moreover, $H$ is intransitive, whereas
$C_n=\langle c\rangle$ is transitive, so
\[
  \chi_{\Q[S_n/H]}(c)
  =
  |(S_n/H)^{\langle c\rangle}|
  =0.
\]
Thus the stated equality holds for every proper subgroup $H<S_n$.

By Lemmas~\ref{lem:J-in-N} and~\ref{lem:omega-J},
\[
  J_n\subseteq\ker\bigl(\omega_n|_{N_n}\bigr).
\]
Conversely, take an arbitrary
\[
  x\in\ker\bigl(\omega_n|_{N_n}\bigr).
\]
Thus $x\in N_n$ and $\omega_n(x)=0$. Since $n$ is composite, fix a
factorization $n=ab$ with $a,b>1$, and set
\[
  \mathcal R:=\mathcal R_{a,b}=-\Theta_{W_{a,b}}.
\]
By Lemma~\ref{lem:wreath-unit}, we may write
\[
  \mathcal R=[S_n/S_n]+u,
\]
where
\[
  u\in
  \left\langle
    [S_n/H]
    \mathrel{}\middle|\mathrel{}
    H\in\mathcal H_n^{\mathrm{np}}
  \right\rangle_{\Z}.
\]
Lemma~\ref{lem:zero-support} gives
\[
  x\in
  \left\langle
    [S_n/H]
    \mathrel{}\middle|\mathrel{}
    H\in\mathcal H_n^{\mathrm{np}}
  \right\rangle_{\Z}.
\]
Since $\mathcal R\in K(S_n)$, Lemma~\ref{lem:projection}, applied to the
Burnside basis expansion of $x$, gives
\[
  x\mathcal R\in J_n.
\]
On the other hand, $x\in N_n\subseteq K(S_n)$, and applying the same lemma to
the Burnside basis expansion of $u$ gives
\[
  xu\in J_n.
\]
Since $[S_n/S_n]$ is the identity element of $\Omega(S_n)$, we have
\[
  x\mathcal R=x+xu.
\]
Therefore,
\[
  x=x\mathcal R-xu\in J_n.
\]
This proves
\[
  J_n=\ker\bigl(\omega_n|_{N_n}\bigr).
\]

Finally, Corollary~\ref{cor:surjective} gives
\[
  \omega_n(N_n)=\Z.
\]
The first isomorphism theorem therefore yields
\[
  O_n
  =N_n/J_n
  =N_n/\ker\bigl(\omega_n|_{N_n}\bigr)
  \xrightarrow{\sim}\Z.
\]
\end{proof}

\section{Comparison with the primitive quotient}\label{sec:primitive-comparison}

\subsection{The comparison exact sequence}

Recall that
\[
  I_n:=I(S_n),
  \qquad
  \Prim(S_n)=K(S_n)/I_n.
\]
Since $J_n\subseteq I_n$, the inclusion $N_n\hookrightarrow K(S_n)$ induces a
homomorphism
\[
  \rho_n:O_n\longrightarrow\Prim(S_n).
\]

\begin{proposition}\label{prop:difference-sequence}
There is a natural exact sequence
\[
  0\longrightarrow\frac{N_n\cap I_n}{J_n}
  \longrightarrow O_n\xrightarrow{\rho_n}\Prim(S_n)
  \longrightarrow\frac{K(S_n)}{N_n+I_n}\longrightarrow0.
\]
\end{proposition}

\begin{proof}
The kernel of $\rho_n$ is $(N_n\cap I_n)/J_n$. Its image is
$(N_n+I_n)/I_n$, and hence its cokernel is $K(S_n)/(N_n+I_n)$.
\end{proof}

The leftmost term measures the relations that survive in the Young--wreath
quotient but become trivial after allowing relations induced from arbitrary
proper subgroups or inflated from proper quotient groups. The rightmost term
measures the part of the primitive quotient that cannot be represented by
standard relations associated with action-imprimitive subgroups.

\subsection{Application of the Bartel--Dokchitser classification}

To apply the classification theorem of Bartel--Dokchitser, we record the
relevant conditions on quotients of symmetric groups. Recall that a finite
group is called $p$-quasi-elementary if it has a cyclic normal subgroup of
order prime to $p$ whose quotient is a $p$-group.

\begin{lemma}\label{lem:BD-application}
The following statements hold.
\begin{enumerate}[label=\textup{(\roman*)},leftmargin=3em,labelsep=0.6em]
\item If $n\ge5$, then $S_n$ is not quasi-elementary, and every proper quotient
      $S_n/N$ with $1\ne N\triangleleft S_n$ is cyclic.
\item The group $S_4$ is not quasi-elementary. Its proper quotients are, up to
      isomorphism,
      \[
        1,\qquad C_2,\qquad S_3.
      \]
      All of them are $2$-quasi-elementary, and among them only $S_3$ is
      noncyclic.
\end{enumerate}
\end{lemma}

\begin{proof}
Suppose first that $n\ge5$, and let $N\triangleleft S_n$. Then
$N\cap A_n\triangleleft A_n$. By the simplicity of $A_n$, the intersection
$N\cap A_n$ is either $1$ or $A_n$.

In the former case, the commutator subgroup $[N,A_n]$ satisfies
\[
  [N,A_n]\le N\cap A_n=1.
\]
Hence $N$ centralizes $A_n$, so
\[
  N\le C_{S_n}(A_n),
\]
where $C_{S_n}(A_n)$ denotes the centralizer of $A_n$ in $S_n$. This
centralizer is trivial: an element centralizing $A_n$ centralizes every
$3$-cycle, which forces it to fix every point. Hence $N=1$. In the latter
case, $A_n\le N$. Since $A_n$ has
index $2$ in $S_n$, we have $N=A_n$ or $N=S_n$. Consequently, the only
normal subgroups of $S_n$ are
\[
  1,\qquad A_n,\qquad S_n,
\]
and the proper quotients by nontrivial normal subgroups are $C_2$ and $1$.

For $n=4$, the normal subgroups are
\[
  1,\qquad V_4,\qquad A_4,\qquad S_4,
\]
and the corresponding proper quotients by nontrivial normal subgroups are
$S_3$, $C_2$, and $1$. The groups $1$ and $C_2$ are $2$-groups and hence
are $2$-quasi-elementary, while
\[
  S_3=C_3\rtimes C_2
\]
is $2$-quasi-elementary and noncyclic.

Finally, $S_n$ has no nontrivial cyclic normal subgroup for $n\ge4$, and
$|S_n|=n!$ is not a prime power. Thus, if $S_n$ were $p$-quasi-elementary,
its cyclic normal subgroup would have to be trivial, forcing $S_n$ itself to
be a $p$-group, a contradiction. Hence $S_n$ is not quasi-elementary.
\end{proof}

By Bartel--Dokchitser's Theorem~4.3, for a non-quasi-elementary group
$G$, one has $\Prim(G)\cong\Z$ if every proper quotient of $G$ is cyclic, and
$\Prim(G)\cong\Z/p\Z$ if every proper quotient is $p$-quasi-elementary for
the same prime $p$ and at least one of them is noncyclic. In both cases,
$\Prim(G)$ is generated by any relation $\Theta$ satisfying
\[
  \coeff_{[G/G]}(\Theta)=1.
\]
See \cite[Theorem~4.3]{BartelDokchitserBrauer}. Lemma~\ref{lem:BD-application}
therefore gives
\begin{equation}\label{eq:BD-values}
  \Prim(S_n)\cong
  \begin{cases}
    \Z/2\Z,&n=4,\\
    \Z,&n\ge5.
  \end{cases}
\end{equation}

Recall that
\[
  O_n:=N_n/J_n,
\]
where $N_n$ has $\Z$-basis consisting of the standard relations $\Theta_H$
indexed by the conjugacy classes of action-imprimitive subgroups that are not
conjugate to Young subgroups, and
\[
  J_n=J_n^{\mathrm Y}+J_n^{\mathrm W}
\]
is the submodule generated by relations induced from proper Young subgroups
and standard wreath subgroups.

\begin{primitivecomparisonrestated}
Let $n\ge4$.
\begin{enumerate}[label=\textup{(\roman*)},leftmargin=3em,labelsep=0.6em]
\item If $n\ge6$ is composite, then
\[
  \rho_n:O_n\xrightarrow{\sim}\Prim(S_n)
\]
is an isomorphism.
\item If $n=4$, then, under the identifications
\[
  O_4\cong\Z,
  \qquad
  \Prim(S_4)\cong\Z/2\Z,
\]
the map $\rho_4$ is reduction modulo $2$.
\item If $n\ge5$ is prime, then
\[
  O_n=0,
  \qquad
  \Prim(S_n)\cong\Z.
\]
\end{enumerate}
\end{primitivecomparisonrestated}

\begin{proof}
Suppose first that $n\ge6$ is composite, and fix a factorization $n=ab$ with
$a,b>1$. By Lemma~\ref{lem:wreath-one},
\[
  \omega_n\bigl(-\Theta_{W_{a,b}}\bigr)=-1.
\]
Hence Main Theorem~\ref{thm:main} shows that the class of
$-\Theta_{W_{a,b}}$ generates $O_n\cong\Z$. By Lemma~\ref{lem:wreath-unit},
\[
  \coeff_{[S_n/S_n]}\bigl(-\Theta_{W_{a,b}}\bigr)=1.
\]
Bartel--Dokchitser's Theorem~4.3 therefore shows that its image in
$\Prim(S_n)\cong\Z$ is a generator. Thus $\rho_n$ is an isomorphism.

For $n=4$, the same argument shows that the class of $-\Theta_{W_{2,2}}$
generates $O_4\cong\Z$ and that its image generates
$\Prim(S_4)\cong\Z/2\Z$. Under the isomorphism
\[
  \overline\omega_4:O_4\xrightarrow{\sim}\Z
\]
and the identification of $\Prim(S_4)$ with $\Z/2\Z$ given by
\[
  [x]\longmapsto \coeff_{[S_4/S_4]}(x)\pmod 2,
\]
the sign in the definition of $\omega_4$ is immaterial. The map $\rho_4$ is therefore reduction modulo
$2$.

Finally, let $n\ge5$ be prime. By the discussion in
Subsection~\ref{subsec:action-imprimitive}, a transitive imprimitive subgroup
of $S_n$ would preserve a block system with blocks of size $a$ and with $b$
blocks, giving a factorization
\[
  n=ab
  \qquad
  (a,b>1).
\]
Since $n$ is prime, no such subgroup exists. Hence every subgroup indexing the basis of $N_n$ is intransitive. By
Proposition~\ref{prop:intransitive-in-Y}, each corresponding standard
relation belongs to $J_n$, and hence
\[
  N_n\subseteq J_n.
\]
Together with Lemma~\ref{lem:J-in-N}, this gives $N_n=J_n$ and therefore
$O_n=0$. On the other hand, \eqref{eq:BD-values} gives
$\Prim(S_n)\cong\Z$.
\end{proof}

\begin{corollary}\label{cor:lattice-equalities}
For every composite $n\ge6$,
\[
  N_n\cap I_n=J_n,
  \qquad
  K(S_n)=N_n+I_n.
\]
For $n=4$,
\[
  \frac{N_4\cap I_4}{J_4}
  =\overline\omega_4^{-1}(2\Z)
  \subseteq O_4,
  \qquad
  K(S_4)=N_4+I_4.
\]
In particular, the restriction of $\overline\omega_4$ induces an
isomorphism
\[
  \frac{N_4\cap I_4}{J_4}
  \xrightarrow{\sim}
  2\Z.
\]
\end{corollary}

\begin{proof}
Apply Proposition~\ref{prop:difference-sequence}. If $n\ge6$ is composite,
Main Theorem~\ref{thm:primitive-comparison} says that $\rho_n$ is an
isomorphism, so both its kernel and cokernel vanish. This gives
\[
  N_n\cap I_n=J_n,
  \qquad
  K(S_n)=N_n+I_n.
\]

For $n=4$, the same theorem identifies $\rho_4$ with reduction modulo $2$
under the isomorphism
\[
  \overline\omega_4:O_4\xrightarrow{\sim}\Z.
\]
Hence
\[
  \ker(\rho_4)
  =\overline\omega_4^{-1}(2\Z).
\]
Since Proposition~\ref{prop:difference-sequence} identifies
\[
  \ker(\rho_4)
  =\frac{N_4\cap I_4}{J_4},
\]
we obtain
\[
  \frac{N_4\cap I_4}{J_4}
  =\overline\omega_4^{-1}(2\Z)
  \subseteq O_4.
\]
The restriction of $\overline\omega_4$ therefore gives an isomorphism from
this subgroup onto $2\Z$. Moreover, $\rho_4$ is surjective, so its cokernel
vanishes. The exact sequence then yields
\[
  K(S_4)=N_4+I_4.
\]
\end{proof}

\section{An extension to the monomial Burnside ring over \texorpdfstring{$C_2$}{C2}}
\label{sec:monomial}

\subsection{The standard basis and linearization}

Let
\[
  C_2=\{\pm1\}.
\]
For a finite group $G$, we write $\Omega^{\pm}(G)$ for the monomial Burnside
ring over $C_2$. Its additive group has a standard $\Z$-basis indexed by the
$G$-conjugacy classes of \emph{monomial pairs}
\[
  (H,\chi),
  \qquad
  H\le G,
  \qquad
  \chi\in\Hom(H,C_2).
\]
We denote the corresponding basis element by $[H,\chi]_G$. If
${}^gK=gKg^{-1}$ and ${}^g\psi(gkg^{-1})=\psi(k)$, multiplication is given by
\[
  [H,\chi]_G[K,\psi]_G
  =
  \sum_{HgK\in H\backslash G/K}
  \left[
    H\cap{}^gK,
    \left.\chi\right|_{H\cap{}^gK}
    \cdot
    \left.{}^g\psi\right|_{H\cap{}^gK}
  \right]_G.
\]
The identity element is $[G,1]_G$. This is the $C_2$-case of Dress's
monomial Burnside ring~\cite{DressMonomial}.

For $x\in\Omega^{\pm}(G)$,
\[
  \coeff_{[H,\chi]_G}(x)
\]
denotes the coefficient of $[H,\chi]_G$ in the standard basis expansion of
$x$.

For a character $\chi:H\to C_2\subseteq\Q^\times$, let $\Q_\chi$ denote the
one-dimensional rational representation of $H$ on which $h\in H$ acts by
\[
  h\cdot q=\chi(h)q
  \qquad(q\in\Q).
\]
The linearization map
\[
  \ell_G^{\pm}:\Omega^{\pm}(G)\longrightarrow R_{\Q}(G),
  \qquad
  [H,\chi]_G\longmapsto[\Ind_H^G\Q_\chi]
\]
is a ring homomorphism. Put
\[
  K^{\pm}(G):=\ker\ell_G^{\pm}.
\]

For $L\le G$, we use the usual induction and restriction maps. On standard
basis elements they are given by
\[
  \Ind_L^G([H,\chi]_L)=[H,\chi]_G
  \qquad(H\le L)
\]
and
\[
  \Res_L^G([H,\chi]_G)
  =
  \sum_{LgH\in L\backslash G/H}
  \left[
    L\cap{}^gH,
    \left.{}^g\chi\right|_{L\cap{}^gH}
  \right]_L.
\]
They commute with the corresponding induction and restriction maps on
rational representation rings. They also satisfy the projection formula
\begin{equation}\label{eq:monomial-projection-formula}
  \Ind_L^G\bigl(x\,\Res_L^G(y)\bigr)
  =
  \Ind_L^G(x)\,y
  \qquad
  \bigl(x\in\Omega^{\pm}(L),\ y\in\Omega^{\pm}(G)\bigr).
\end{equation}
These properties follow directly from the standard basis formulas above; see
also Dress~\cite{DressMonomial}.

Attaching the trivial character defines a natural injective ring homomorphism
\[
  \iota_G:\Omega(G)\longrightarrow\Omega^{\pm}(G),
  \qquad
  [G/H]\longmapsto[H,1]_G.
\]
It satisfies
\[
  \ell_G^{\pm}\circ\iota_G=\ell_G.
\]

\subsection{The Young--wreath quotient}

For $S_n$, the map $\iota_{S_n}$ identifies the Young partial Burnside ring
$P_n$ with
\[
  P_n^{\pm,0}
  :=
  \iota_{S_n}(P_n)
  =
  \bigoplus_{\lambda\vdash n}
  \Z[S_\lambda,1]_{S_n}
  \subseteq\Omega^{\pm}(S_n).
\]
Since $\ell_{S_n}^{\pm}\circ\iota_{S_n}=\ell_{S_n}$, the Young section induces
\[
  \sigma_n^{\pm}
  :=
  \iota_{S_n}\circ\sigma_n:
  R_{\Q}(S_n)\longrightarrow P_n^{\pm,0}.
\]
We therefore define
\[
  \Theta_{H,\chi}^{\pm}
  :=
  [H,\chi]_{S_n}
  -
  \sigma_n^{\pm}\bigl([\Ind_H^{S_n}\Q_\chi]\bigr).
\]
For brevity, we call $(H,\chi)$ a \emph{trivial-character Young pair} if
$H$ is conjugate to a Young subgroup and $\chi=1$. The elements
$\Theta_{H,\chi}^{\pm}$ indexed by pairs that are not trivial-character Young
pairs form a $\Z$-basis of $K^{\pm}(S_n)$. This follows from the same
direct-sum argument as Proposition~\ref{prop:standard-basis}.

Choose a set $\mathfrak R_n^{\pm}$ of representatives of the
$S_n$-conjugacy classes of monomial pairs $(H,\chi)$ such that
$H\in\mathcal H_n^{\mathrm{np}}$ and $(H,\chi)$ is not a trivial-character
Young pair. Define
\[
  N_n^{\pm}
  :=
  \bigoplus_{(H,\chi)\in\mathfrak R_n^{\pm}}
  \Z\Theta_{H,\chi}^{\pm}
\]
and
\[
  J_n^{\pm}
  :=
  \sum_{\substack{\lambda\vdash n\\\lambda\ne(n)}}
  \Ind_{S_\lambda}^{S_n}K^{\pm}(S_\lambda)
  +
  \sum_{\substack{ab=n\\a,b>1}}
  \Ind_{W_{a,b}}^{S_n}K^{\pm}(W_{a,b}).
\]
Finally, define
\[
  \omega_n^{\pm}:K^{\pm}(S_n)\longrightarrow\Z,
  \qquad
  \omega_n^{\pm}(x)
  :=
  -\coeff_{[S_n,1]_{S_n}}(x).
\]

\begin{lemma}\label{lem:monomial-J-in-N}
We have
\[
  J_n^{\pm}\subseteq N_n^{\pm}.
\]
\end{lemma}

\begin{proof}
A standard basis element induced from a proper Young subgroup or from a
standard wreath subgroup has the form $[H,\chi]_{S_n}$, where $H$ is
conjugate to a subgroup of the source of induction. Hence $H$ is
action-imprimitive.

For each $\Theta_{H,\chi}^{\pm}$, the only standard basis term outside the
span of the trivial-character Young pairs is $[H,\chi]_{S_n}$. The same
coordinate comparison as in the proof of Lemma~\ref{lem:J-in-N} therefore
shows that no coefficient corresponding to a transitive primitive subgroup can
occur in an element of $J_n^{\pm}$.
\end{proof}

Consequently, the quotient
\[
  O_n^{\pm}:=N_n^{\pm}/J_n^{\pm}
\]
is well defined.

\begin{lemma}\label{lem:monomial-nonprimitive-induction}
If $H\in\mathcal H_n^{\mathrm{np}}$, then
\[
  \Ind_H^{S_n}K^{\pm}(H)\subseteq J_n^{\pm}.
\]
\end{lemma}

\begin{proof}
If $H$ is intransitive, then it is contained in a proper Young subgroup. If
$H$ is transitive imprimitive, then it is contained in a conjugate of a
standard wreath subgroup. The result follows from the compatibility of the
linearization maps with induction and from transitivity of
induction, exactly as in Lemma~\ref{lem:all-nonprimitive-induction}.
\end{proof}

\begin{lemma}\label{lem:monomial-projection}
Let $H\in\mathcal H_n^{\mathrm{np}}$, let $\chi:H\to\{\pm1\}$, and let
$z\in K^{\pm}(S_n)$. Then
\[
  [H,\chi]_{S_n}z\in J_n^{\pm}.
\]
\end{lemma}

\begin{proof}
The projection formula in $\Omega^{\pm}$ gives
\[
  [H,\chi]_{S_n}z
  =
  \Ind_H^{S_n}
  \bigl([H,\chi]_H\,\Res_H^{S_n}(z)\bigr).
\]
The element in parentheses belongs to $K^{\pm}(H)$. The assertion now follows
from Lemma~\ref{lem:monomial-nonprimitive-induction}.
\end{proof}

\begin{lemma}\label{lem:monomial-zero-span}
If $x\in N_n^{\pm}$ and $\omega_n^{\pm}(x)=0$, then
\[
  x\in
  \left\langle
    [H,\chi]_{S_n}
    \mathrel{}\middle|\mathrel{}
    H\in\mathcal H_n^{\mathrm{np}},\
    \chi:H\to\{\pm1\}
  \right\rangle_{\Z}.
\]
\end{lemma}

\begin{proof}
By the definition of $N_n^{\pm}$, the element $x$ has a unique expansion
\[
  x
  =
  \sum_{(H,\chi)\in\mathfrak R_n^{\pm}}
  m_{H,\chi}\Theta_{H,\chi}^{\pm}
  \qquad
  (m_{H,\chi}\in\Z).
\]
Every relation occurring in this sum has the form
\[
  \Theta_{H,\chi}^{\pm}
  =
  [H,\chi]_{S_n}
  -
  \sigma_n^{\pm}\bigl([\Ind_H^{S_n}\Q_\chi]\bigr),
  \qquad
  H\in\mathcal H_n^{\mathrm{np}}.
\]
The Young-section term is a $\Z$-linear combination of trivial-character
Young pairs. Every proper Young subgroup is intransitive. Hence the only
standard basis element that may occur without having the form
$[H,\chi]_{S_n}$ with $H\in\mathcal H_n^{\mathrm{np}}$ is
$[S_n,1]_{S_n}$. By assumption,
\[
  \coeff_{[S_n,1]_{S_n}}(x)
  =
  -\omega_n^{\pm}(x)
  =
  0.
\]
The assertion follows.
\end{proof}

\begin{lemma}\label{lem:monomial-wreath-unit}
Let $n=ab$ with $a,b>1$, and set
\[
  \mathcal R_{a,b}^{\pm}
  :=
  -\Theta_{W_{a,b},1}^{\pm}.
\]
Then there exists
\[
  u\in
  \left\langle
    [H,\chi]_{S_n}
    \mathrel{}\middle|\mathrel{}
    H\in\mathcal H_n^{\mathrm{np}},\
    \chi:H\to\{\pm1\}
  \right\rangle_{\Z}
\]
such that
\[
  \mathcal R_{a,b}^{\pm}=[S_n,1]_{S_n}+u.
\]
\end{lemma}

\begin{proof}
The relation $\Theta_{W_{a,b},1}^{\pm}$ is the trivial-character copy of the
ordinary relation $\Theta_{W_{a,b}}$. Hence Lemma~\ref{lem:wreath-one} gives
\[
  \omega_n^{\pm}(\Theta_{W_{a,b},1}^{\pm})=1,
\]
and therefore
\[
  \coeff_{[S_n,1]_{S_n}}(\mathcal R_{a,b}^{\pm})=1.
\]
Moreover, $W_{a,b}$ is transitive imprimitive, while every proper Young
subgroup is intransitive. After separating the term $[S_n,1]_{S_n}$, all
remaining standard basis elements therefore have the required form.
\end{proof}

\begin{theorem}
\label{thm:monomial-main}
Let $n\ge4$ be composite. Then
\[
  J_n^{\pm}
  =
  \ker\bigl(\omega_n^{\pm}|_{N_n^{\pm}}\bigr),
  \qquad
  O_n^{\pm}\cong\Z.
\]
For every proper subgroup $H<S_n$ and every character
$\chi:H\to\{\pm1\}$,
\begin{equation}\label{eq:monomial-character}
  \omega_n^{\pm}(\Theta_{H,\chi}^{\pm})
  =
  \chi_{\Ind_H^{S_n}\Q_\chi}(c)
  =
  \sum_{\substack{gH\in S_n/H\\g^{-1}cg\in H}}
  \chi(g^{-1}cg).
\end{equation}
\end{theorem}

\begin{proof}
Evaluating the Young-section expansion at $c$ and applying the induced
character formula gives \eqref{eq:monomial-character}. An element induced
from a proper subgroup has zero coefficient at $[S_n,1]_{S_n}$, so
\[
  J_n^{\pm}\subseteq\ker\omega_n^{\pm}.
\]
Choose a factorization $n=ab$ with $a,b>1$. Then
\[
  \omega_n^{\pm}(\Theta_{W_{a,b},1}^{\pm})=1,
\]
so $\omega_n^{\pm}|_{N_n^{\pm}}$ is surjective.

Conversely, let
\[
  x\in\ker\bigl(\omega_n^{\pm}|_{N_n^{\pm}}\bigr).
\]
By
Lemma~\ref{lem:monomial-wreath-unit},
\[
  \mathcal R_{a,b}^{\pm}
  =
  [S_n,1]_{S_n}+u,
\]
where $u$ is a $\Z$-linear combination of the elements $[H,\chi]_{S_n}$
with $H\in\mathcal H_n^{\mathrm{np}}$. By
Lemma~\ref{lem:monomial-zero-span}, the same is true of $x$.
Since $\mathcal R_{a,b}^{\pm}\in K^{\pm}(S_n)$ and
$x\in K^{\pm}(S_n)$, Lemma~\ref{lem:monomial-projection}, applied termwise,
gives
\[
  x\mathcal R_{a,b}^{\pm}\in J_n^{\pm},
  \qquad
  xu\in J_n^{\pm}.
\]
Because $[S_n,1]_{S_n}$ is the identity of $\Omega^{\pm}(S_n)$,
\[
  x
  =
  x\mathcal R_{a,b}^{\pm}-xu
  \in J_n^{\pm}.
\]
Thus
\[
  J_n^{\pm}
  =
  \ker\bigl(\omega_n^{\pm}|_{N_n^{\pm}}\bigr).
\]
Surjectivity and the first isomorphism theorem now give
$O_n^{\pm}\cong\Z$.
\end{proof}

\subsection{Reduction to the ordinary Burnside ring}

For the embedding $\iota_G:\Omega(G)\to\Omega^{\pm}(G)$ defined above,
define an additive map in the opposite direction
\begin{equation}\label{eq:red-map}
  \red_G:\Omega^{\pm}(G)\longrightarrow\Omega(G)
\end{equation}
on standard basis elements by
\[
  \red_G([H,1]_G)=[G/H],
\]
and, for $\chi\ne1$, by
\[
  \red_G([H,\chi]_G)
  =
  [G/\ker\chi]-[G/H].
\]
We call $\red_G$ the \emph{index-two permutation reduction}: for nontrivial
$\chi$, the subgroup $\ker\chi$ has index $2$ in $H$, and the basis element
$[H,\chi]_G$ is replaced by a difference of permutation $G$-sets. This map
need not be multiplicative. If $K=\ker\chi$, then, as rational characters of
$H$,
\[
  \Ind_K^H\one_K=\one_H+\chi,
\]
where $\one_H$ denotes the trivial character of $H$. Hence
\[
  \ell_G\circ\red_G=\ell_G^{\pm}.
\]
Consequently,
\[
  \red_G(K^{\pm}(G))\subseteq K(G),
\]
and a direct calculation on basis elements shows that $\red_G$ commutes with
induction.

\begin{lemma}\label{lem:red-preserves-lattices}
For the embedding $\iota_G:\Omega(G)\to\Omega^{\pm}(G)$ defined above, we
have
\[
  \red_{S_n}(N_n^{\pm})\subseteq N_n,
  \qquad
  \red_{S_n}(J_n^{\pm})\subseteq J_n,
\]
and
\[
  \iota_{S_n}(N_n)\subseteq N_n^{\pm},
  \qquad
  \iota_{S_n}(J_n)\subseteq J_n^{\pm}.
\]
\end{lemma}

\begin{proof}
The last two inclusions and the inclusion
$\red_{S_n}(J_n^{\pm})\subseteq J_n$ follow directly from the definitions and
from compatibility with induction. If $\chi\ne1$ and $K=\ker\chi$, then
\[
  \red_{S_n}(\Theta_{H,\chi}^{\pm})
  =
  \Theta_K-\Theta_H,
\]
whereas
\[
  \red_{S_n}(\Theta_{H,1}^{\pm})=\Theta_H.
\]
If $H\in\mathcal H_n^{\mathrm{np}}$, then $K\le H$ either remains
transitive imprimitive with the same block system or is intransitive. Thus
$K\in\mathcal H_n^{\mathrm{np}}$, proving
$\red_{S_n}(N_n^{\pm})\subseteq N_n$.
\end{proof}

\begin{theorem}
\label{thm:monomial-ordinary}
For every composite integer $n\ge4$, the maps $\red_{S_n}$ and $\iota_{S_n}$
induce mutually inverse isomorphisms
\[
  \overline{\red}_{S_n}:O_n^{\pm}\xrightarrow{\sim}O_n,
  \qquad
  \overline{\iota}_{S_n}=\overline{\red}_{S_n}^{-1}.
\]
Moreover,
\[
  J_n^{\pm}\cap\iota_{S_n}(N_n)=\iota_{S_n}(J_n),
  \qquad
  \omega_n\circ\red_{S_n}=\omega_n^{\pm}
  \quad\text{on }N_n^{\pm}.
\]
\end{theorem}

\begin{proof}
Lemma~\ref{lem:red-preserves-lattices} shows that both maps descend to the
quotients, and
\[
  \red_{S_n}\circ\iota_{S_n}=\mathrm{id}.
\]

Suppose that $\chi\ne1$ and put $K=\ker\chi$. Then
\begin{equation}\label{eq:local-monomial-relation}
  \Theta_{H,\chi}^{\pm}
  -\Theta_{K,1}^{\pm}
  +\Theta_{H,1}^{\pm}
  =
  [H,\chi]_{S_n}-[K,1]_{S_n}+[H,1]_{S_n}.
\end{equation}
The right-hand side is induced from the relation
\[
  [H,\chi]_H-[K,1]_H+[H,1]_H\in K^{\pm}(H).
\]
If $H\in\mathcal H_n^{\mathrm{np}}$, then
Lemma~\ref{lem:monomial-nonprimitive-induction} shows that
\eqref{eq:local-monomial-relation} belongs to $J_n^{\pm}$. Hence
$\iota_{S_n}\circ\red_{S_n}$ is the identity modulo $J_n^{\pm}$ on every
basis element of $N_n^{\pm}$. The two induced maps are therefore
mutually inverse.

If $x\in J_n^{\pm}\cap\iota_{S_n}(N_n)$, then
$\red_{S_n}(x)\in J_n$ and
\[
  x=\iota_{S_n}\bigl(\red_{S_n}(x)\bigr)\in\iota_{S_n}(J_n).
\]
This proves one inclusion in the intersection formula; the reverse inclusion
is immediate. Finally, for the trivial character the compatibility of the two
obstructions is the ordinary fixed-coset formula. If $\chi\ne1$, it follows by
evaluating
\[
  \red_{S_n}(\Theta_{H,\chi}^{\pm})=\Theta_K-\Theta_H
\]
and
\[
  \Ind_K^H\one_K=\one_H+\chi
\]
at $c$. Thus
\[
  \omega_n\circ\red_{S_n}=\omega_n^{\pm}
  \quad\text{on }N_n^{\pm}.
\]
\end{proof}

\begin{corollary}
If $n\ge6$ is composite, then $O_n^{\pm}$ also recovers the primitive
quotient, and the natural composite
\[
  O_n^{\pm}\xrightarrow{\sim}O_n\xrightarrow{\sim}\Prim(S_n)
\]
is an isomorphism. For $n=4$, this composite corresponds to reduction modulo
$2$ from $\Z$ to $\Z/2\Z$.
\end{corollary}

\end{document}